\documentclass[reqno,openany,dvipdfmx]{amsart}
\usepackage{amsthm,amsmath,amssymb,amsthm}
\usepackage[left=3.5cm, right=3.5cm, includefoot]{geometry}
\usepackage[all]{xy}
\usepackage{color}
\usepackage{appendix}
\usepackage{tikz}

\theoremstyle{definition}
\newtheorem{theorem}{Theorem}[section]
\newtheorem{proposition}{Proposition}[section]
\newtheorem{lemma}{Lemma}[section]
\newtheorem{corollary}{Corollary}[section]

\newtheorem{example}{Example}[section]
\newtheorem{definition}{Definition}[section]
\newtheorem{remark}{Remark}[section]

\numberwithin{equation}{section}

\begin{document}
\title{Characters of infinite-dimensional quantum classical groups: $BCD$ cases}
\author[R. Sato]{Ryosuke SATO}
\address{Graduate School of Mathematics, Nagoya University, Chikusaku, Nagoya 464-8602, Japan}
\email{d19001r@math.nagoya-u.ac.jp}
\maketitle

\begin{abstract}
We study the character theory of inductive limits of $q$-deformed classical compact groups. In particular, we clarify the relationship between the representation theory of Drinfeld--Jimbo quantized universal enveloping algebras and our previous work on the quantized characters. We also apply the character theory to construct Markov semigroups on unitary duals of $SO_q(2N+1)$, $Sp_q(N)$, and their inductive limits.
\end{abstract}

\allowdisplaybreaks{
%%%%%%%%%%%%%%%%%%%%%%%%%%%%%%%%%%%%%%%%%%%%%%%%%%%%%%%%%%%%%%%%%%%%%%%%%%%%%%%%%%%%%%%%%%%%%%%%%%%%%%%%%%%%%%%%%%%%%%%%%%%%%%%%%%
\section{Introduction}
The study of characters is one of the crucial topics in representation theory. For a topological group $G$ a \emph{character} of $G$ is a positive-definite continuous function on $G$ satisfying $\chi(gh)=\chi(hg)$ for any $g, h\in G$ and $\chi(e)=1$. Then the classification of irreducible representations of a compact group, up to unitarily equivalence, is equivalent to that of extreme characters. More generally, the classification of finite factor representations of topological groups, up to quasi-equivalence, is equivalent to that of extreme characters. For the infinite-dimensional classical groups $U(\infty)=\varinjlim_N U(N)$, $SO(\infty)=\varinjlim_N SO(N)$, and $Sp(\infty)=\varinjlim_NSp(N)$, there are many studies of characters. See \cite{Voiculescu76}, \cite{VK82}, \cite{Boyer83}, \cite{Boyer92}, \cite{OO06}, ..., etc. In particular, Vershik and Kerov \cite{VK82} initiated a stochastic approach to studying characters of such infinite-dimensional groups. For instance, the study of characters of $U(\infty)$ is essentially equivalent to the study of stochastic objects related to the \emph{Gelfand-Tsetlin graph}. Then the description of all extreme characters is equivalent to that of the \emph{minimal boundary} of the Gelfand--Tsetlin graph.

After that, Gorin \cite{Gorin12} introduced the \emph{$q$-Gelfand--Tsetlin graph} motivated by the representation theory of $U_q\mathfrak{gl}_N$, and he described its minimal (and Martin) boundary. Moreover, in our previous work \cite{Sato1}, the author gave the explicit representation-theoretic meaning of Gorin's work in terms of compact quantum groups. In those works \cite{Gorin12}, \cite{Sato1}, the quantization parameter $q\in (0,1)$ appeared as follows: It is well known that $\mathbb{S}_N:=\{\lambda=(\lambda_1\geq\cdots\geq\lambda_N)\in\mathbb{Z}\}$ gives a complete list of (equivalence classes of) irreducible representations of $U(N)$ and its $q$-deformed compact quantum group $U_q(N)$. By the Wely character formula, for any $\lambda\in\mathbb{S}_N$
\[\frac{s_\lambda(z_1,\dots,z_N)}{s_\lambda(1,\dots,1)}\]
is equal to the value of the associated irreducible character of $U(N)$ at $\mathrm{diag}(z_1,\dots,z_N)\in U(N)$, where $s_\lambda$ is the Schur polynomial. When we replace $U(N)$ by $U_q(N)$, this formula changes as
\[\frac{s_\lambda(q^{N-1}z_1, q^{N-3}z_2,\dots, q^{-N+1}z_N)}{s_\lambda(q^{N-1}, q^{N-3},\dots, q^{-N+1})},\]
which is also the quantum trace of the irreducible representation of $U_q\mathfrak{gl}_N$ with the highest weight $\lambda$. Now we define the function ${}_q\Lambda^N_{N-1}$ on $\mathbb{S}_N\times \mathbb{S}_{N-1}$ by
\[\frac{s_\lambda(q^{N-1}z_,\dots, q^{-N+3}z_{N-1},q^{-Z+1})}{s_\lambda(q^{N-1}, q^{N-3},\dots, q^{-N+1})}=\sum_{\mu\in\mathbb{S}_{N-1}}{}_q\Lambda^N_{N-1}(\lambda, \mu)\frac{s_\mu(q^{N-2}z_1,\dots, q^{-N+2}z_{N-1})}{s_\mu(q^{N-2},\dots, q^{-N+2})}\]
for any $\lambda\in\mathbb{S}_N$. Then we can show that ${}_q\Lambda^N_{N-1}(\lambda, \,\cdot\,)\geq0$ and $\sum_{\mu\in\mathbb{S}_{N-1}}{}_q\Lambda^N_{N-1}(\lambda,\mu)=1$ for any $\lambda\in\mathbb{S}_N$. Let $\mathcal{M}_p(\mathbb{S}_N)$ be the set of all probability measures on $\mathbb{S}_N$. Then we obtain the mapping $m\in\mathcal{M}_p(\mathbb{S}_N)\mapsto m{}_q\Lambda^N_{N-1}\in\mathcal{M}_p(\mathbb{S}_{N-1})$ by
\[m{}_q\Lambda^N_{N-1}(\mu):=\sum_{\lambda\in\mathbb{S}_N}m(\lambda){}_q\Lambda^N_{N-1}(\mu, \lambda)\]
for any $\mu\in\mathbb{S}_{N-1}$. Namely, the $\mathcal{M}_p(\mathbb{S}_N)$ form a projective system and an element in the projective limit $\varprojlim_N(\mathcal{M}_p(\mathbb{S}_N),{}_q\Lambda^N_{N-1})$ is called a \emph{coherent system} of the $q$-Gelfand--Tsetlin graph. We remark that ${}_q\Lambda^N_{N-1}$ can also be defined as $q\to 1$. Then the set of all characters of $U(\infty)$ is affine homeomorphic to $\varprojlim_N(\mathcal{M}_p(\mathbb{S}_N), {}_1\Lambda^N_{N-1})$. In particular, the set of \emph{extreme} characters of $U(\infty)$ is homeomorphic to the set of extreme points in $\varprojlim_N(\mathcal{M}_p(\mathbb{S}_N), {}_1\Lambda^N_{N-1})$, which is called the \emph{minimal boundary} of the Gelfand--Tsetlin graph.

For $q\in (0, 1)$ we can extend those results. The author \cite{Sato3} introduced the notion of inductive limits of compact quantum groups based on the operator algebraic formulation of quantum groups due to Masuda and Nakagami \cite{MasudaNakagami}. Using this notion, we can define the inductive limit $U_q(\infty)$ of quantum unitary groups $U_q(N)$. Then the set of all \emph{quantized} characters of $U_q(\infty)$ is affine homeomorphic to $\varprojlim_N(\mathcal{M}_p(\mathbb{S}_N), {}_q\Lambda^N_{N-1})$. In particular, the set of \emph{extreme} quantized characters of $U_q(\infty)$ is homeomorphic to the minimal boundary of the $q$-Gelfand--Tsetlin graph. See \cite{Sato1}, \cite{Sato3} for more details.

The purpose of this paper is to clarify the relation between our previous work mentioned above and the representation theory of Drinfeld--Jimbo quantized universal enveloping algebras, which are algebraic quantizations of Lie groups. In this paper, we study the character theory of inductive limits of $q$-deformed classical compact groups of types $B, C$, and $D$ based on the representation theory of $U_q\mathfrak{so}_{2N+1}, U_q\mathfrak{sp}_{2N}, U_q\mathfrak{so}_{2N}$. As a byproduct, we also give an explicit representation-theoretic interpretation to the study of the coherent systems of the \emph{BC type $q$-Gelfand--Tsetlin graph} due to Cuenca and Gorin in \cite{CG20}. We also apply the character theory in the paper to construct Markov semigroups on unitary duals of compact quantum groups $SO_q(2N+1)$, $Sp_q(N)$, and their inductive limits. Then we obtain new examples of Markov semigroups which are studied in the paper \cite{Sato4}.

The organization of the paper is following: Section 2 and 3 serve to prepare basic facts on compact quantum groups and Drinfeld--Jimbo quantized universal enveloping algebras, respectively. In Section 4, we study the character theory of $q$-deformed compact quantum groups $SO_q(N)$ and $Sp_q(N)$. In Section 5, we study the character theory of inductive limits $SO_q(2\infty+1)$, $Sp_q(\infty)$ of $SO_q(2N+1)$, $Sp_q(N)$, respectively, and we give an explicit representation-theoretic interpretation to the coherent systems of the BC type $q$-Gelfand--Tsetlin graph in Theorem \ref{thm:main}. Moreover, we give complete parametrizations of extreme quantized characters of $SO_q(2\infty+1)$ and $Sp_q(\infty)$ in Corollary \ref{cor:para}. In Section 6, we construct Markov semigroups on unitary duals of $SO_q(2N+1)$, $Sp_q(N)$, and their inductive limits. See Theorem \ref{thm:markov}.

%%%%%%%%%%%%%%%%%%%%%%%%%%%%%%%%%%%%%%%%%%%%%%%%%%%%%%%%%%%%%%%%%%%%%%%%%%%%%%%%%%%%%%%%%%%%%%%%%%%%%%%%%%%%%%%%%%%%%%%%%%%%%%%%%%
\section{Compact quantum groups}
We review basic facts on compact quantum groups. Here we mainly refer to \cite[Chapter 1]{NeshveyevTuset}.

Let $G=(C(G),\delta_G)$ be a pair of unital $C^*$-algebra $C(G)$ and unital $*$-homomorphism $\delta_G\colon C(G)\to C(G)\otimes C(G)$, where $\otimes$ denotes the operation of minimal tensor product of $C^*$-algebras. If $G$ satisfies that
\begin{itemize}
\item $(\delta_G\otimes\mathrm{id})\delta_G=(\mathrm{id}\otimes\delta_G)\delta_G$ as $*$-homomorphism from $C(G)$ to $C(G)\otimes C(G)\otimes C(G)$,
\item $(C(G)\otimes1)\delta_G(C(G)), (1\otimes C(G))\delta_G(C(G))$ are dense in $C(G)\otimes C(G)$,
\end{itemize}
then $G$ is called a \emph{compact quantum group}. Let $\mathcal{H}$ be a finite-dimensional Hilbert space and $U$ a unitary element in $B(\mathcal{H})\otimes C(G)$. If $(\mathrm{id}\otimes \delta_G)(U)=U_{12}U_{13}$, where $U_{12}, U_{13}$ are the leg numbering notations, then $U$ is called a \emph{unitary corepresentation} of $G$. For any linear functional $\varphi$ on $B(\mathcal{H})$ the element $(\varphi\otimes\mathrm{id})(U)$ in $C(G)$ is called a \emph{matrix coefficient} of $U$. Let $\mathbb{C}[G]$ be the subspace of $C(G)$ consisting of all matrix coefficients of finite-dimensional unitary corepresentations. Then it is known that $\mathbb{C}[G]$ becomes a $*$-subalgebra of $C(G)$. Conversely, we always assume that $C(G)$ is the universal $C^*$-algebra generated by $\mathbb{C}[G]$ throughout the paper. Moreover, $\mathbb{C}[G]$ becomes a Hopf $*$-algebra. We use the same symbol $\delta_G$ to denote its comultiplication. Then $\mathbb{C}[G]^*$ becomes a $*$-algebra by the multiplication given by $xy:=(x\otimes y)\delta_G$ for any $x, y\in\mathbb{C}[G]^*$.

Let $\widehat{G}$ be the set of all equivalence classes of irreducible unitary corepresentations of $G$. We fix the complete representatives $\{U_\lambda\}_{\lambda\in\widehat{G}}$ and let $\mathcal{H}_\lambda$ be the representation spaces of the $U_\lambda$. We define $\Phi\colon \mathbb{C}[G]^*\to \prod_{\lambda\in\widehat{G}}B(\mathcal{H}_\lambda)$ by $\Phi(x):=((\mathrm{id}\otimes x)(U_\lambda))_{\lambda\in\widehat G}$ for any $x\in\mathbb{C}[G]^*$. Then $\Phi$ is a $*$-isomorphism. Using this $*$-isomorphism $\Phi$, we define two subalgebras $C^*(G)$ and $W^*(G)$ of $\mathbb{C}[G]^*$ as the images of
\[c_0\mathchar`-\bigoplus_{\lambda\in\widehat{G}}B(\mathcal{H}_\lambda):=\left\{(x_\lambda)_{\lambda\in\widehat{G}}\in\prod_{\lambda\in\widehat{G}}B(\mathcal{H}_\lambda)\,\middle|\,\lim_{\lambda\in\widehat G}\|x_\lambda\|=0\right\},\]
\[\ell^\infty\mathchar`-\bigoplus_{\lambda\in\widehat{G}}B(\mathcal{H}_\lambda):=\left\{(x_\lambda)_{\lambda\in\widehat{G}}\in\prod_{\lambda\in\widehat{G}}B(\mathcal{H}_\lambda)\,\middle|\,\sup_{\lambda\in\widehat G}\|x_\lambda\|<\infty\right\}\]
by the $\Phi^{-1}$, respectively. Here we remark that the $\|\,\cdot\,\|$ are the operator norms of the $B(\mathcal{H}_\lambda)$, and $\lim_{\lambda\in\widehat{G}}\|x_\lambda\|=0$ if for any $\epsilon>0$ there exists a finite subset $F\subset \widehat G$ such that $\|x_\lambda\|<\epsilon$ for $\lambda\in\widehat G\backslash F$. Then $C^*(G)$ (resp. $W^*(G)$) is a $C^*$-algebra (resp. a von Neumann algebra), called the \emph{group $C^*$-algebra} (resp. the \emph{group von Neumann algebra}) of $G$.

We denote by $\{f^G_z\}_{z\in\mathbb{C}}$ the distinguished family of linear functionals on $\mathbb{C}[G]$, called the \emph{Woronowicz characters}. They induce an action $\tau^G\colon \mathbb{C}\curvearrowright\mathbb{C}[G]$, called the \emph{scaling action}, by
\[\tau_z(a):=(f^G_{-\mathrm{i}z}\otimes\mathrm{id}\otimes f^G_{\mathrm{i}z})((\delta_G\otimes\mathrm{id})(\delta_G(a)))\]
for any $a\in\mathbb{C}[G]$ and $z\in\mathbb{C}$. Let $\widehat \tau^G\colon\mathbb{C}\curvearrowright \mathbb{C}[G]^*$ be the dual action, that is, $\widehat\tau^G_z(x):=x\circ\tau^G_z$ for any $x\in\mathbb{C}[G]^*$ and $z\in\mathbb{C}$. Then it is known that $\Phi\widehat\tau^G_z\Phi^{-1}=\prod_{\lambda\in\widehat G}\mathrm{Ad}F_\lambda^{\mathrm{i}z}$ for any $z\in\mathbb{C}$, where $F_\lambda:=(\mathrm{id}\otimes f^G_z)(U_\lambda)$, which is positive and invertible in $B(\mathcal{H}_\lambda)$. In particular, $(\widehat\tau^G_t)_{t\in\mathbb{R}}$ preserves $C^*(G)$ and $W^*(G)$. In what follows, we use the same symbol to denote the restrictions of $(\widehat\tau^G_t)_{t\in\mathbb{R}}$ to $C^*(G)$ and $W^*(G)$.

\begin{definition}
A normal $\widehat\tau^G$-KMS state on $W^*(G)$ with the inverse temperature -1 is called a \emph{quantized character} of $G$. Let $\mathrm{Ch}(G)$ be the set of quantized characters of $G$.
\end{definition}

\begin{remark}
If $G$ is a compact group and $C(G)$ is the commutative $C^*$-algebra of continuous functions on $G$, then $\widehat\tau^G$ is trivial and quantized characters are nothing less than normal tracial states on $W^*(G)$. Moreover, they correspond to the characters of $G$ in a natural way.
\end{remark}

\begin{remark}
For any $\lambda\in \widehat G$ we define the linear functional $\chi_\lambda$ on $W^*(G)$ by 
\[\chi_\lambda(\Phi^{-1}(x_\beta)_{\beta\in\widehat G}):=\frac{\mathrm{Tr}_{\mathcal{H}_\lambda}(F_\lambda x_\lambda)}{\mathrm{Tr}_{\mathcal{H}_\lambda}(F_\lambda)}.\] Then $\chi_\lambda$ becomes a quantized character of $G$. We call it the \emph{irreducible quantized character} associated with $\lambda\in\widehat G$. Then every quantized character of $G$ is a convex combination of irreducible ones. See \cite[Lemma 2.2]{Sato1}, \cite[Proposition 2.1]{Sato3}.
\end{remark}

%%%%%%%%%%%%%%%%%%%%%%%%%%%%%%%%%%%%%%%%%%%%%%%%%%%%%%%%%%%%%%%%%%%%%%%%%%%%%%%%%%%%%%%%%%%%%%%%%%%%%%%%%%%%%%%%%%%%%%%%%%%%%%%%%%
\section{Drinfeld--Jimbo quantized universal enveloping algebras}
Here we review basic facts on Drinfeld--Jimbo quantized universal enveloping algebras. See \cite[Section 10.1]{CP:book}, \cite[Chapter 6, 7]{KliSch} for more details. We fix a parameter $0<q<1$ throughout the paper. Let $\mathfrak{g}$ be a complex semisimple Lie algebra and $\mathfrak{h}$ a Cartan subalgebra of $\mathfrak{g}$. We denote by $A=(a_{ij})_{i,j=1}^N$ its Cartan matrix and let $D=\mathrm{diag}(d_1,\dots, d_N)$ be the diagonal matrix such that $DA$ is symmetric and positive-definite. Then the \emph{Drinfeld--Jimbo quantized universal enveloping algebra} $U_q\mathfrak{g}$ is the associative unital complex universal algebra generated by $E_i, F_i, K_i, K_i^{-1}$ ($i=1,\dots, N$) satisfying that
\[K_iK_j=K_jK_i,\quad K_iK_i^{-1}=K_i^{-1}K_i=1,\]
\[K_iE_jK_i^{-1}=q_i^{a_{ij}}E_j,\quad K_iF_jK_i^{-1}=q_i^{-a_{ij}}F_j,\]
\[E_iF_j-F_jE_i=\delta_{i,j}\frac{K_i-K_i^{-1}}{q_i-q_i^{-1}},\]
\[\sum_{r=0}^{1-a_{ij}}(-1)^r\begin{bmatrix}1-a_{ij}\\r\end{bmatrix}_{q_i}E_i^{1-a_{ij}-r}E_jE^r_i=0,\quad i\neq j,\]
\[\sum_{r=0}^{1-a_{ij}}(-1)^r\begin{bmatrix}1-a_{ij}\\r\end{bmatrix}_{q_i}F_i^{1-a_{ij}-r}F_jF^r_i=0,\quad i\neq j,\]
where $q_i:=q^{d_i}$ and 
\[ [n]_q=\frac{q^{n}-q^{-n}}{q-q^{-1}},\quad \begin{bmatrix}n\\r\end{bmatrix}_q:=\frac{[n]_q[n-1]_q\cdots[1]_q}{[r]_q\cdots[1]_q[n-r]_q\cdots[1]_q}.\]
Then it is known that $U_q\mathfrak{g}$ becomes a Hopf $*$-algebra. %We denote by $\Delta$ its comultiplication.% counit $\epsilon$, antipode $S$, and $*$-operation are given by
%\[\Delta(K^{\pm 1}_i)=K^{\pm1}_i\otimes K^{\pm1}_i,\]
%\[\Delta(E_i)=E_i\otimes K_i+1\otimes E_i,\quad \Delta(F_i)=F_i\otimes1+K_i^{-1}\otimes F_i,\]
%\[\epsilon(K_i)=1,\quad \epsilon(E_i)=\epsilon(F_i)=0,\]
%\[S(K_i)=K_i^{-1},\quad S(E_i)=-E_iK_i^{-1},\quad S(F_i)=-K_iF_i,\]
%\[K_i^*=K_i,\quad E_i^*=K_iF_i,\quad F_i^*=E_iK_i^{-1}.\]

Let $\Delta$ and $\Delta_+$ be the sets of roots and positive roots of $\mathfrak{g}$, respectively. We remark that $\Delta_+$ depends on the choice of ordered basis $H_1,\dots, H_N\in\mathfrak{h}$. In what follows, we denote by $\epsilon_1,\dots, \epsilon_N\in\mathfrak{h}^*$ the dual basis, that is, $\epsilon_i(H_j)=\delta_{i,j}$ for any $i, j=1,\dots, N$. Let $\alpha_1,\dots,\alpha_N$ be the simple roots of $\mathfrak{g}$. Then $\alpha_1,\dots,\alpha_N$ are also linearly independent and we have a symmetric bilinear form on $\mathfrak{h}^*$ such that
$(\alpha_i, \alpha_j):=d_ia_{ij}$
for any $i,j=1,\dots,N$. We denote by $Q$ the root lattice, that is, $Q=\sum_{i=1}^N\mathbb{Z}\alpha_i$. For $\alpha=\sum_{i=1}^Nn_i\alpha_i\in Q$ we define $K_\alpha:=K_1^{n_1}\cdots K_N^{n_N}$. %Then it is known that $S^2(X)=K_{2\rho}XK_{2\rho}^{-1}$, where $2\rho=\sum_{\alpha\in \Delta_+}\alpha$.

Let $P_+$ be the set of dominant weights, that is, $P_+=\{\lambda\in \mathfrak{h}^*\mid (\lambda,\alpha_i^\vee)\in\mathbb{Z}_{\geq0},\,i=1,\dots, N\}$, where $\alpha_i^\vee=2\alpha_i/(\alpha_i,\alpha_i)$. Then every finite-dimensional type 1 irreducible representation of $U_q\mathfrak{g}$ is equivalent to a representation with a highest weight in $P_+$. Conversely, every element in $P_+$ is a highest weight of unique, up to equivalence, finite-dimensional type 1 irreducible representation of $U_q\mathfrak{g}$. In what follows, we denote by $(T_\lambda, V_\lambda)$ the finite-dimensional type 1 irreducible representation of $U_q\mathfrak{g}$ with the highest weight $\lambda\in P_+$. We remark that $(T_\lambda, V_\lambda)$ is a $*$-representation with respect to an appropriate inner product of $V_\lambda$.

In the paper, the linear functional on $B(V_\lambda)$ given as $\mathrm{Tr}_{V_\lambda}(K_{2\rho}\,\cdot\,)$ plays a prominent role, where $2\rho=\sum_{\alpha\in\Delta_+}^N\alpha$. We describe such linear functionals when $\mathfrak{g}$ is one of the $\mathfrak{gl}_N$, $\mathfrak{so}_{2N+1}$, $\mathfrak{sp}_{2N}$, and $\mathfrak{so}_{2N}$, where they are the Lie algebras of $GL(n,\mathbb{C})$, $SO(2N+1, \mathbb{C}), Sp(2N, \mathbb{C}), SO(2N,\mathbb{C})$. Namely, they are complexifications of Lie algebras of compact Lie groups $U(N)$, $SO(2N+1)$, $Sp(N)$, $SO(2N)$, respectively. Here we remark that $Sp(N):= Sp(2N,\mathbb{C})\cap U(2N)$ and
\[Sp(2N,\mathbb{C}):=\{A\in SL(2N, \mathbb{C})\mid {}^tAJA=J\},\quad J=\begin{pmatrix}0&I_N\\-I_N&0\end{pmatrix}.\]
%First we recall the following basic facts:
%\begin{itemize}
%\item[($A_N$)] When $\mathfrak{g}=\mathfrak{gl}_N$, we have $2\rho=\sum_{i=1}^N(N-2i+1)\epsilon_i$ and
%\[P_+=\left\{\sum_{i=1}^N\lambda_i\epsilon_i\,\middle|\,\lambda_1\geq\cdots\geq\lambda_N,\quad (\lambda_1,\dots,\lambda_N)\in\mathbb{Z}^N\right\}.\]
%
%\item[($B_N$)] When $\mathfrak{g}=\mathfrak{so}_{2N+1}$, we have $2\rho=\sum_{i=1}^N(N-i+1/2)\epsilon_i$ and 
%\[P_+=\left\{\sum_{i=1}^N\lambda_i\epsilon_i\,\middle|\, \lambda_1\geq\cdots\geq\lambda_N,\quad (\lambda_1,\cdots,\lambda_N)\in(\mathbb{Z}_{\geq0})^N\cup\left(\mathbb{Z}_{\geq0}+\frac{1}{2}\right)^N\right\}.\]
%
%\item[($C_N$)] When $\mathfrak{g}=\mathfrak{sp}_{N}$, we have $2\rho=\sum_{i=1}^N(N-i+1)\epsilon_i$ and 
%\[P_+=\left\{\sum_{i=1}^N\lambda_i\epsilon_i\,\middle|\, \lambda_1\geq\cdots\geq\lambda_N,\quad (\lambda_1,\cdots,\lambda_N)\in(\mathbb{Z}_{\geq0})^N\right\}.\]
%
%\item[($D_N$)] When $\mathfrak{g}=\mathfrak{so}_{2N}$, we have $2\rho=\sum_{i=1}^N(N-i)\epsilon_i$ and 
%\[P_+=\left\{\sum_{i=1}^N\lambda_i\epsilon_i\,\middle|\, \lambda_1\geq\cdots\geq|\lambda_N|,\quad (\lambda_1,\cdots,\lambda_N)\in(\mathbb{Z})^N\cup\left(\mathbb{Z}\frac{1}{2}\right)^N\right\}.\]
%\end{itemize}

\noindent
We define constants
\begin{equation}\label{eq:daniels}
\epsilon(X):=\begin{cases}1/2&(X=B),\\ 1&(X=C),\\0&(X=D),\end{cases}\quad c(X)_i:=i-1+\epsilon(X)\quad(i\geq1).
\end{equation}
Then, $\mathfrak{g}=\mathfrak{so}_{2N+1}$, $\mathfrak{sp}_{2N}$ and $\mathfrak{so}_{2N}$ we may choose a system of positive roots such that $2\rho=\sum_{i=1}^Nc(X)_i\epsilon_i$, where $X=B, C, D$, respectively.

For $z_1,\dots,z_N\in\mathbb{C}\backslash\{0\}$ we define
\[V(z_1,\dots,z_N):=\prod_{1\leq i<j\leq N}(z_j-z_i),\quad V^s(z_1,\dots,z_N):=\prod_{1\leq i<j\leq N}(z_i+z_i^{-1}-z_j-z_j^{-1}).\]
Then we have
\[V^s(z_1,\dots,z_N)=\frac{1}{2}\det[z_j^{c(D)_{N-i+1}}+z_j^{-c(D)_{N-i+1}}]_{i, j=1}^N=\det\left[\frac{z_j^{c(X)_{N-i+1}}-z_j^{-c(X)_{N-i+1}}}{z_j^{\epsilon(X)}-z_j^{-\epsilon(X)}}\right]_{i, j=1}^N\]
if $X=B,C$. See \cite[Appendix A.3]{FH:book}. 

The following is a consequence of the Weyl character formula:
\begin{lemma}\label{lem:miya} 
Let $\lambda=\sum_{i=1}^N\lambda_i\epsilon_i\in P_+$, $\alpha=\sum_{i=1}^Na_i\epsilon_i\in Q$ and $z_1=q^{a_1},\dots,z_N=q^{a_N}$. If $\mathfrak{g}=\mathfrak{gl}_N$, then we have
\begin{equation}\label{eq:skyA}
\mathrm{Tr}_{V_\lambda}(T_\lambda(K_\alpha))=s_\lambda(z_1,\dots,z_N):=\frac{\det[z_j^{\lambda_i+N-i}]_{i, j=1}^N}{V(z_1,\dots,z_N)}.
\end{equation}
If $\mathfrak{g}=\mathfrak{so}_{2N+1}, \mathfrak{sp}_{2N}$ or $\mathfrak{so}_{2N}$, then we have
\begin{align}\label{eq:sky}&\mathrm{Tr}_{V_\lambda}(T_\lambda(K_\alpha))=V^s(z_1,\dots,z_N)^{-1}\times \nonumber
\\&\begin{cases}
\displaystyle \det\left[\frac{z_j^{\lambda_i+c(X)_{N-i+1}}-z_j^{-\lambda_i-c(X)_{N-i+1}}}{z_j^{\epsilon(X)}-z_j^{-\epsilon(X)}}\right]_{i, j=1}^N\quad(\mathfrak{g}=\mathfrak{so}_{2N+1}, X=B\text{ or }\mathfrak{g}=\mathfrak{sp}_{2N}, X=C),\\
%\displaystyle \frac{\det[q^{a_j(\lambda_i+c*C_N)_i)}-q^{-a_j(\lambda_i+c(C_N)_i)}]_{i,j=1}^N}{\det[q^{a_jc(C_N)_i}-q^{-a_jc(C_N)_i}]_{i,j=1}^N}&(\mathfrak{g}=\mathfrak{sp}_{2N}),\\
\displaystyle \frac{1}{2}(\det[z_j^{\lambda_i+c(D)_{N-i+1}}+z_j^{-\lambda_i-c(D)_{N-i+1}}]_{i, j=1}^N+\det[z_j^{\lambda_i+c(D)_{N-i+1}}-z_j^{-\lambda_i-c(D)_{N-i+1}}]_{i,j=1}^N)\quad(\mathfrak{g}=\mathfrak{so}_{2N}).
\end{cases}\end{align}
\end{lemma}
\begin{proof}
Let $V_\lambda=\bigoplus_{\mu\in Q(\lambda)}V_{\lambda;\mu}$ be the weight decomposition of $V_\lambda$, where $Q(\lambda)\subset Q$ and $V_{\lambda; \mu}$ is a weight subspace of $V_\lambda$ with the weight $\mu\in Q(\lambda)$. Then we have
\[\mathrm{Tr}_{V_\lambda}(K_\alpha)=\sum_{\mu\in Q(\lambda)}\dim V_{\lambda;\mu}q^{(\alpha, \mu)}=\sum_{\mu=\sum_{i=1}^N\mu_i\epsilon \in Q(\lambda)}\dim V_{\lambda;\mu}z_1^{\mu_1}\cdots z_N^{\mu_N}.\]
Since $V_\lambda$ has the same weight decomposition of the irreducible representation of $\mathfrak{g}$ with the highest weight $\lambda$ (see \cite[Proposition 7.11]{KliSch}), we obtain the Equation \eqref{eq:skyA}, \eqref{eq:sky} by the Weyl character formula of $\mathfrak{g}$ (see \cite[Chapter 24]{FH:book}).
\end{proof}

We define a function $f^{X_N}_\lambda(z_1,\dots,z_N)$ as the right-hand side in Equation \eqref{eq:skyA}, \eqref{eq:sky} for $X=A, B, C, D$. For instance, we have $\mathrm{Tr}_{V_\lambda}(K_{2\rho})=f^{X_N}_\lambda(q^{c(X)_1},\dots, q^{c(X)_N})$ for $X=B, C, D$ and $f^{A_N}_\lambda=s_\lambda$.

\begin{remark}
If $\mathfrak{g}=\mathfrak{so}_{2N}$ and $\lambda=\sum_{i=1}^N\lambda_i\epsilon_i\in P_+$, then we have $\tilde{\lambda}=\sum_{i=1}^{N-1}\lambda_i\epsilon_i-\lambda_N\epsilon_N\in P_+$ and, by the direct computation,
\[f^{D_N}_\lambda(z_1,\dots,z_N)+f^{D_N}_{\tilde{\lambda}}(z_1,\dots,z_N)=\frac{\det[z_j^{\lambda_i+c(D)_{N-i+1}}+z_j^{-\lambda_i-c(D)_{N-i+1}}]_{i, j=1}^N}{V^s(z_1,\dots,z_N)}.\]
\end{remark}

%%%%%%%%%%%%%%%%%%%%%%%%%%%%%%%%%%%%%%%%%%%%%%%%%%%%%%%%%%%%%%%%%%%%%%%%%%%%%%%%%%%%%%%%%%%%%%%%%%%%%%%%%%%%%%%%%%%%%%%%%%%%%%%%%%
\section{The FRT construction and quantized characters}
First, we briefly review basic facts on the FRT(Faddeev--Reshetikhin--Takhtajan) construction of compact quantum groups. Here we refer to \cite[Chapter 9, 11]{KliSch}. See also \cite{FRT}, \cite{Hayashi92}, \cite{Takeuchi}.

For $\mathfrak{g}=\mathfrak{so}_{2N+1}, \mathfrak{sp}_{2N}, \mathfrak{so}_{2N}$ let $\mathcal{U}_q=U_{q^{1/2}}\mathfrak{so}_{2N+1}, U_q\mathfrak{sp}_{2N}, U_q\mathfrak{so}_{2N}$ and $R\in\mathcal{U}_q\otimes\mathcal{U}_q$ the \emph{universal $R$-matrix} of $\mathcal{U}_q$. Let $(T_1, V_1)$ be the finite-dimensional type 1 irreducible representation of $\mathcal{U}_q$ with the highest weight $\lambda$ given by $(\lambda,\alpha_i^\vee)=-\delta_{i,1}$ for $i=1,\dots, N$, where $\alpha_1,\dots, \alpha_N$ are simple roots of $\mathfrak{g}$. We remark that $\lambda$ is a dominant weight with respect to the ordered sequence $-\alpha_1,\dots,-\alpha_N$ (see \cite[Section 8.4.1]{KliSch}). Then we use the same symbol $R$ to denote $(T_1\otimes T_1)(R)$. We fix a matrix unit system $\{e_{ij}\}_{i,j=1}^d$ of $\mathrm{End}(V_1)$, where $d:=\dim V_1=2N+1, 2N, 2N$, respectively, and let $R=\sum_{i,j,k,l=1}^dR_{kl}^{ij}e_{ik}\otimes e_{jl}$. Then the FRT algebra $\mathcal{A}(R)$ is the universal algebra generated by $u_{ij}$ ($i,j=1,\dots,d$) with the relations:
\begin{equation}\label{eq:FRT}\sum_{k,l=1}^dR_{kl}^{ji}u_{km}u_{ln}=\sum_{k,l=1}^du_{ik}u_{jl}R_{mn}^{lk}\end{equation}
for any $i, j, m, n=1,\dots,d$. We remark that the FRT algebra $\mathcal{A}(R)$ does not depend on the choice of a matrix unit system of $\mathrm{End}(V_1)$. For any $i=1,\dots, N$ we define
\[c_i:=q^{-c(X)_{N-i+1}}
\quad
c_{d-i+1}:=\begin{cases}
c_i^{-1}&(\mathfrak{g}=\mathfrak{so}_{2N+1}, \mathfrak{so}_{2N}),\\
-c_i^{-1}&(\mathfrak{g}=\mathfrak{sp}_{2N}),
\end{cases}\]
and $c_{N+1}:=1$ if $\mathfrak{g}=\mathfrak{so}_{2N+1}$. Then we define three algebras $\mathcal{O}(SO_q(2N+1)),\mathcal{O}(Sp_q(N))$ and $\mathcal{O}(SO_q(2N))$ as quotients of $\mathcal{A}(R)$ by the following ways:

\begin{itemize}
\item[($B_N, D_N$):] The algebra $\mathcal{O}(SO_q(d))$ is generated by $u_{ij}$ ($i, j=1,\dots, d=2N+1, 2N$) satisfying Equation \eqref{eq:FRT} and 
\[\sum_{n=1}^dc_{d-j+1}c_nu_{nj}u_{d-n+1,d-j+1}=1\quad (j=1,\dots, d),\]
\[\sum_{\sigma\in S(d)}(-1)^{l(\sigma)}u_{1\sigma(1)}\cdots u_{d\sigma(d)}=1.\]

\item[($C_N$):] The algebra $\mathcal{O}(Sp_q(N))$ is generated by $u_{ij}$ ($i, j=1,\dots, d=2N$) satisfying Equation \eqref{eq:FRT} and 
\[\sum_{n=1}^dc_{d-j+1}c_nu_{nj}u_{d-n+1,d-j+1}=-1\quad (j=1,\dots, d).\]
\end{itemize}

Then $\mathcal{O}(SO_q(2N+1)),\mathcal{O}(Sp_q(N))$ and $\mathcal{O}(SO_q(2N))$ are Hopf $*$-algebras. Moreover, it is known that there exist compact quantum groups $G_q=SO_q(2N+1), Sp_q(N), SO_q(2N)$ such that $\mathbb{C}[G_q]=\mathcal{O}(G_q)$, respectively. See \cite[Chapter 11]{KliSch}.

\begin{remark}
The generators $u_{ij}$ ($i, j=1,\dots,d$) of $\mathcal{O}(G_q)$ depend on the choice of matrix unit system $(e_{ij})_{i,j=1}^d$ of the fundamental representation $(T_1, V_1)$. Now we define linear functionals $t_{ij}$ on $\mathcal{U}_q$ for any $i, j=1,\dots, d$ by $T_1(X)=\sum_{i,j=1}^d t_{ij}(X)e_{ij}$ for any $X\in \mathcal{U}_q$. Then we have the dual pairing of two Hopf $*$-algebras $\mathcal{U}_q$ and $\mathcal{O}(G_q)$ by $(X, u_{ij})=t_{ij}(X)$ for any $X\in\mathcal{U}_q$ and $i, j=1,\dots, d$.
\end{remark}

\begin{remark}
There exists a similar construction of the quantum unitary groups $U_q(N)$ (see \cite[Chapter 9]{KliSch}). Then two Hopf $*$-algebras $U_q\mathfrak{gl}_N$ and $\mathbb{C}[U_q(N)]$ have a dual pairing.
\end{remark}

For any finite-dimensional unitary correpresentation $U$ of $G_q$ on $\mathcal{H}$ we have the representation $(T_U, \mathcal{H})$ of $\mathcal{U}_q$ by $T_U(X)=(\mathrm{id}\otimes X)(U)$, where we regard $X\in\mathcal{O}(G_q)^*$ by the above dual pairing. Now we define 
$P(G_q):=\left\{\lambda=\sum_{i=1}^N\lambda_i\epsilon_i\in P_+\,\middle|\, \lambda_i\in\mathbb{Z}\right\}$.
Then, by \cite[Theorem 11.22]{KliSch}, for any irreducible corepresentation $U$ of $G_q$ on $\mathcal{H}$, the corresponding representation $(T_U,\mathcal{H})$ is equivalent to $(T_\lambda, V_\lambda)$ for some $\lambda\in P(G_q)$. Conversely, for any $\lambda\in P(G_q)$ there exists a unique, up to equivalence, irreducible corepresentation $U$ of $G_q$ on $\mathcal{H}$ such that $(T_U, \mathcal{H})$ is equivalent to $(T_\lambda, V_\lambda)$. Namely, we have $\widehat{G_q}\cong P(G_q)$ as sets. In what follows, we denote by $U_\lambda$ the irreducible corepresentation associated with $\lambda\in P(G_q)$. 

Let $\epsilon_i=\sum_{j=1}^Nc_{ij}\alpha_j$ and $L_{\lambda ,i}:=T_\lambda(K_1)^{c_{i1}}\cdots T_\lambda(K_N)^{c_{iN}}$ for any $i=1,\dots, N$. Now we define the map $\kappa\colon\mathbb{T}^N\to \mathcal{O}(G_q)^*$ by 
$\kappa(q^{\mathrm{i}t_1},\dots, q^{\mathrm{i}t_N}):=\Phi^{-1}((L_{\lambda, 1}^{\mathrm{i}t_1}\cdots L_{\lambda,N}^{\mathrm{i}t_N})_{\lambda\in P(G_q)})$
for any $t_i\in[0, -2\pi/\log q)$. We remark that $L_{\lambda, i}^{\mathrm{i}t}$ is unitary for any $t\in\mathbb{R}$ since $L_{\lambda, i}$ is selfadjoint. Thus, $\kappa(z_1,\dots, z_N)\in W^*(G_q)$ for any $(z_1,\dots,z_N)\in\mathbb{T}^N$. Moreover, for any weight vector $v_\mu\in V_\lambda$ with a weight $\mu\in Q$, we have $L_{\lambda, 1}^{\mathrm{i}t_1}\cdots L_{\lambda,N}^{\mathrm{i}t_N}v_\mu=q^{\mathrm{i}(t_1\epsilon_1+\cdots+t_N\epsilon_N, \mu)}v_\mu$.

\begin{proposition}\label{prop:contact}
For each $G_q=SO_q(2N+1), Sp_q(N), SO_q(2N)$ and any $\lambda\in P(G_q)$
\[\chi_\lambda(\kappa(z_1,\dots,z_N))=\frac{f^{X_N}_\lambda(q^{c(X)_1}z_1,\dots,q^{c(X)_N}z_N)}{f^{X_N}_\lambda(q^{c(X)_1},\dots, q^{c(X)_N})},\]
where $X=B, C, D$, respectively.
\end{proposition}
\begin{proof}
Using the dual pairing of $\mathcal{U}_q$ and $\mathcal{O}(G_q)$, the Woronowicz character $f_1$ on $\mathcal{O}(G_q)$ is given as $(K_{2\rho}, \,\cdot\,)$. See \cite[Section 11.3.4]{KliSch}. Thus, for any $t_1,\dots, t_N\in [0,-2\pi/\log q)$ we have 
\[\chi_\lambda(\kappa(q^{\mathrm{i}t_1},\dots, q^{\mathrm{i}t_N}))=\frac{\mathrm{Tr}_{V_\lambda}(T_\lambda(K_{2\rho})L_{\lambda,1}^{\mathrm{i}t_1}\cdots L_{\lambda,N}^{\mathrm{i}t_N})}{\mathrm{Tr}_{V_\lambda}(T_\lambda(K_{2\rho}))}.\]
Thus, as the same proof of Lemma \ref{lem:miya}, we obtain the statement.
\end{proof}

Since the dual pairing of $\mathcal{O}(G_q)$ and $\mathcal{U}_q$ is nondegenerate (see \cite[Corollary 11.23]{KliSch}), we can identify $\mathcal{O}(G_q)$ with the linear subspace of $\mathcal{U}_q^*$. By \cite[Proposition 7.20]{KliSch}, we have the following:

\begin{lemma}\label{lem:jack}
$\mathcal{O}(G_q)$ is the linear subspace of $\mathcal{U}_q^*$ consisting of all matrix coefficients of subrepresentations in $(T_1^{\otimes n}, V_1^{\otimes n})$ for some $n\in\mathbb{N}$.
\end{lemma}

\section{The infinite-dimensional quantum groups of type $B, C, D$}
In this section, we study the quantized characters of inductive limits of $q$-deformed compact quantum groups of type $B, C, D$. For the case of type $A$ we already have studied in \cite{Sato1}. Let $(\mathcal{U}_{q, N}, G_q(N))$ be one of the 
\[\quad (U_{q^{1/2}}\mathfrak{so}_{2N+1}, SO_q(2N+1)),\quad (U_q\mathfrak{sp}_{2N}, Sp_q(N)),\quad (U_q\mathfrak{so}_{2N}, SO_q(2N)).\] 
We denote by $E_{N,i}, F_{N,i}, K_{N,i}^{\pm1}$ ($i=1,\dots, N$) the generators of $\mathcal{U}_{q, N}$. To describe a Hopf $*$-algebra homomorphism from $\mathcal{U}_{q, N}$ to $\mathcal{U}_{q, N+1}$ clearly, we fix a matrix representation of the Cartan matrix $A=(a_{ij})_{i,j=1}^N$ as follows:
\begin{itemize}
%\item[($A_N$)] When $\mathfrak{g}_N=\mathfrak{gl}_{N}$, we set $a_{ij}=2\delta_{ij}-\delta_{i,j+1}-\delta_{i,j-1}$ for $i,j=1,\dots,N$.

\item[($B_N$)] When $\mathfrak{g}_N=\mathfrak{so}_{2N+1}$, we set $a_{21}=-2$ and otherwise $a_{ij}=2\delta_{ij}-\delta_{i,j+1}-\delta_{i,j-1}$.

\item[($C_N$)] When $\mathfrak{g}_N=\mathfrak{sp}_{2N}$, we set $a_{12}=-2$ and otherwise $a_{ij}=2\delta_{ij}-\delta_{i,j+1}-\delta_{i,j-1}$.

\item[($D_N$)] When $\mathfrak{g}_N=\mathfrak{so}_{2N}$, we set $a_{12}=a_{21}=0$, $a_{31}=a_{13}=-1$, and otherwise $a_{ij}=2\delta_{ij}-\delta_{i,j+1}-\delta_{i,j-1}$.
\end{itemize}

Then we obtain the Hopf $*$-algebra homomorphism from $\Theta_{N}\colon \mathcal{U}_{q,N}\to \mathcal{U}_{q, N+1}$ such that
\[\Theta_{N}(E_{N,i})=E_{N+1,i},\quad \Theta_{N}(F_{N,i})=F_{N+1,i},\quad \Theta_{N}(K_{N,i}^{\pm1})=K_{N+1,i}^{\pm1}\]
for any $i=1,\dots, N$. We denote by the dual map $\theta_N\colon \mathcal{U}_{q, N+1}^*\to \mathcal{U}_{q, N}^*$. 

\begin{lemma}
$G_q(N)$ is a quantum subgroup of $G_q(N+1)$.
\end{lemma}
\begin{proof}
We remark that $G_q(N)$ has the same representation theory of $U(N), SO(2N+1), Sp(N)$ or $SO(2N)$. Thus, it suffices to show that $\theta_N(\mathcal{O}(G_q(N+1)))=\mathcal{O}(G_q(N))$ by \cite[Theorem 2.7.10]{NeshveyevTuset}, \cite[Lemma 2.8]{Tomatsu}. However, for every $\lambda\in P(G_q(N+1))$ the representation $(T_\lambda\Theta_N, V_\lambda)$ of $\mathcal{U}_{q, N}$ is decomposed into irreducible representations with highest weights in $P(G_q(N))$ (see \cite[Section 25.3]{FH:book}). Therefore, by Lemma \ref{lem:jack}, we have $\theta_N(\mathcal{O}(G_q(N+1)))=\mathcal{O}(G_q(N))$.
\end{proof}

By this lemma, the Hopf $*$-algebra homomorphism $\Theta_N\colon\mathcal{U}_{q, N}\to\mathcal{U}_{q, N+1}$ can be extended to the $*$-algebra homomorphism from $\mathcal{O}(G_q(N))^*$ to $\mathcal{O}(G_q(N+1))^*$, and it maps $W^*(G_q(N))$ to $W^*(G_q(N+1))$. In what follows, we use the same symbol $\Theta_N$ to denote this $*$-homomorphism from $W^*(G_q(N))$ to $W^*(G_q(N+1))$. Then $\Theta_N$ is a normal unital $*$-homomorphism.

Let $\kappa_N\colon\mathbb{T}^N\to W^*(G_q(N))$ be as is in Section 4. We remark that $\chi\Theta_N\in\mathrm{Ch}(G_q(N))$ for any $\chi\in\mathrm{Ch}(G_q(N+1))$.

\begin{lemma}\label{lem:yuki}
For any $\chi\in\mathrm{Ch}(G_q(N+1))$ and any $(z_1,\dots,z_N)\in\mathbb{T}^N$
\[\chi(\Theta_{N}(\kappa_N(z_1,\dots,z_N)))=\chi(\kappa_{N+1}(z_1,\dots,z_N,1)).\] 
\end{lemma}
\begin{proof}
For each $N$ there is a $*$-isomorphism $\Phi_{N}\colon\mathcal{O}(G_q(N))^*\to\prod_{\lambda\in P(G_q(N))}B(V_\lambda)$ such that $\Phi_{N}(x):=((\mathrm{id}\otimes x)(U_\lambda))_{\lambda\in P(G_q(N))}$ for any $x\in\mathcal{O}(G_q(N))^*$. For any $\nu\in P(G_q(N+1))$ we let $\bigoplus_{\lambda\in P(\nu)}(T_\lambda, V_\lambda)^{\oplus m^\nu_\lambda}$ be the irreducible decomposition of $(T_\nu\Theta_N, V_\nu)$. Then, by identifying $\bigoplus_{\lambda\in P(\nu)}B(V_\lambda)\otimes M_{m^\nu_\lambda}(\mathbb{C})$ with a subalgebra of $B(V_\nu)$, we have 
\[\Theta_N(\Phi_N^{-1}((A_\lambda)_{\lambda\in P(G_q(N))}))=\Phi_{N+1}^{-1}\left(\left(\sum_{\lambda\in P(\nu)}A_\lambda\otimes I_{m^\nu_\lambda}\right)_{\nu\in P(G_q(N+1))}\right)\]
for any $\Phi_N^{-1}((A_\lambda)_{\lambda\in P(G_q(N))})\in W^*(G_q(N))$. See \cite[Section 2.3]{Sato1}. For $\lambda\in P(G_q(N))$ let $\bigoplus_{\alpha\in Q(\lambda)}V_{\lambda;\alpha}$ be the weight decomposition of $V_\lambda$, where $V_{\lambda; \alpha}$ is the weight subspace with the weight $\alpha$. Then we have $L_{\lambda, i}^{\mathrm{i}t}=\sum_{\alpha\in Q(\lambda)}q^{\mathrm{i}t(\epsilon_i,\alpha)}p_{\lambda;\alpha}$, where $p_{\lambda;\alpha}$ is the projection onto $V_{\lambda;\alpha}$. Moreover, we have $V_\nu\cong\bigoplus_{\lambda\in P(\nu)}\bigoplus_{\alpha\in Q(\lambda)}\bigoplus_{\beta\in Q(\nu)}(V_{\lambda;\alpha}\cap V_{\nu;\beta})^{\oplus m^\nu_\lambda}$, and for any $i=1,\dots, N$ we have $L_{\nu,i}^{\mathrm{i}t}=\sum_{\lambda\in P(\nu)}\sum_{\alpha\in Q(\lambda)\sum_{\beta\in Q(\nu)}}q^{\mathrm{i}t(\epsilon_i, \alpha)}p_{\lambda;\alpha}p_{\nu;\beta}\otimes I_{m^\nu_\lambda}$. Therefore, for $i=1,\dots, N$ we have 
$\Theta_N(\Phi_N^{-1}((L_{\lambda,i}^{\mathrm{i}t})_{\lambda\in P(G_q(N))}))=\Phi_{N+1}^{-1}((L_{\nu,i}^{\mathrm{i}t})_{\nu\in P(G_q(N+1))}))$, and hence 
\begin{align*}
\Theta_N(\kappa_N(q^{\mathrm{i}t_1},\dots, q^{\mathrm{i}t_N}))
%&=\Theta_N(\Phi_N^{-1}((L_{\lambda,1}^{\mathrm{i}t_1}\cdots L_{\lambda,N}^{\mathrm{i}t_N})_{\lambda\in P(G_q(N))}))\\
%&=\Phi_{N+1}^{-1}((L_{\nu,1}^{\mathrm{i}t_1}\cdots L_{\nu, N}^{\mathrm{i}t_N})_{\nu\in P(G_q(N+1))})\\
&=\kappa_{N+1}(q^{\mathrm{i}t_1},\dots, q^{\mathrm{i}t_N}, 1)
\end{align*}
for any $t_1,\dots, t_N\in [0, -2\pi/\log q)$.
\end{proof}

Now we define $G_q(\infty)$ as the inductive limit quantum group of the $G_q(N)$, that is, $G_q(\infty)$ is the tuple $(W^*(G_q(\infty)),\mathfrak{A}, \delta, R, \tau)$ of
\begin{itemize}
\item the $W^*$-algebra $W^*(G_q(\infty))$ given as the inductive limit of the $W^*(G_q(N))$,
\item a $\sigma$-weakly dense $C^*$-subalgebra $\mathfrak{A}\subset W^*(G_q(\infty))$, called \emph{Stratila--Voiculescu AF-algebra},
\item a comultiplication $\delta\colon W^*(G_q(\infty))\to W^*(G_q(\infty))\bar\otimes W^*(G_q(\infty))$, that is, $\delta$ is a unital normal $*$-homomorphism satisfying $(\delta\otimes\mathrm{id})\delta=(\mathrm{id}\otimes\delta)\delta$,
\item a unitary antipode $R\colon W^*(G_q(\infty))\to W^*(G_q(\infty))$, that is, $R$ is an involutive normal anti-$*$-automorphism satisfying $\delta R=\sigma(R\otimes R)\delta^{op}$,
\item a deformation automorphism group $\tau\colon \mathbb{R}\curvearrowright W^*(G_q(\infty))$, that is, $\tau$ is a $\mathbb{R}$-action on $W^*(G_q(\infty))$ satisfying $\delta\tau_t=(\tau_t\otimes \tau_t)\delta$ and $R\tau_t=\tau_t R$ for any $t\in\mathbb{R}$. 
\end{itemize}

We remark that $\delta^{op}:=\sigma\delta$, where $\sigma$ is the flip map on $W^*(G_q(\infty)\otimes W^*(G_q(\infty))$. Since $W^*(G_q(\infty))$ is the inductive limit of the $W^*(G_q(N))$, there exist $\Theta^\infty_N\colon W^*(G_q(N))\to W^*(G_q(\infty))$ for any $N$ such that $\Theta^\infty_{N+1}\Theta_N=\Theta^\infty_N$, and they satisfy the \emph{universality of inductive limits}. See \cite{Sato3} for more details.

\begin{definition}
We define $G_q(\infty)=SO_q(2\infty+1)$, $Sp_q(\infty)$, $SO_q(2\infty)$ as their inductive limits of $G_q(N)=SO_q(2N+1), Sp_q(N), SO_q(2N)$, respectively. 
\end{definition}

\begin{remark}
In the classical case, there is no difference between $SO(2\infty+1)$ and $SO(2\infty)$ because $SO(2N-1)$, $SO(2N)$ are naturally subgroups in $SO(2N)$, $SO(2N+1)$, respectively. However, in the quantum case, there are no natural homomorphisms from $U_q\mathfrak{so}_{2N-1}$ to $U_q\mathfrak{so}_{2N}$ or from $U_q\mathfrak{so}_{2N}$ to $\mathfrak{so}_{2N+1}$ since Cartan matrices of type $B, D$ are quite different. Thus, we have two inductive limits $SO_q(2\infty+1)$ and $SOq(2\infty)$ in the quantum case.
\end{remark}

\begin{definition}
A normal $\tau$-KMS state on $W^*(G_q(\infty))$ with the inverse temperature -1 is called a \emph{quantized character} of $G_q(\infty)$.
\end{definition}

Let $\mathrm{Ch}(G_q(\infty))$ be the set of all quantized characters of $G_q(\infty)$. To analyze $\mathrm{Ch}(G_q(\infty))$, we introduce coherent systems of the BC type $q$-Gelfand-Tsetlin graph. Let $\mathbb{S}_N^+$ be the set of nonnegative signatures, that is,
$\mathbb{S}_N^+=\{\lambda=(\lambda_1,\dots, \lambda_N)\in\mathbb{Z}^N\mid \lambda_1\geq\cdots\geq\lambda_N\geq0\}$. Then we can identify $\mathbb{S}_N^+$ with $P(G_q(N))$ when $G_q(N)=SO_q(2N+1)$ or $Sp_q(N)$. For any $\lambda\in\mathbb{S}_N^+$ we define the function $g^{X_N}_\lambda$ on $\mathbb{T}^N$ by
\[g^{X_N}_\lambda(z_1,\dots, z_N):=\begin{cases}
\displaystyle \frac{f^{X_N}_\lambda(q^{c(X)_1}z_1,\dots, q^{c(X)_N}z_N)}{f^{X_N}_\lambda(q^{c(X)_1},\dots, q^{c(X)_N})} & (X=B, C),\\
\displaystyle \frac{f^{X_N}_\lambda(q^{c(D)_1}z_1,\dots, q^{c(D)_N}z_N)+f^{X_N}_{\tilde{\lambda}}(q^{c(D)_1}z_1,\dots, q^{c(D)_N}z_N)}{f^{X_N}_\lambda(q^{c(D)_1},\dots, q^{c(D)_N})+f^{X_N}_{\tilde{\lambda}}(q^{c(D)_1},\dots, q^{c(D)_N})} & (X=D),\\
\end{cases}\]
where $\tilde{\lambda}:=\sum_{i=1}^{N-1}\lambda_i\epsilon_i-\lambda_N\epsilon_N\in P(SO_q(2N))$ for $\lambda\in\mathbb{S}_N^+$. Let $X$ be one of the $B,C,D$. Then we define a function ${}_q\Lambda^N_{N-1}$ on $\mathbb{S}^+_N\times \mathbb{S}^+_{N-1}$ by
\[g^{X_{N}}_\lambda(z_1,\dots,z_{N-1},1)=\sum_{\mu\in\mathbb{S}_N^+}{}_q\Lambda^{N+1}_N(\lambda, \mu)g^{X_{N-1}}_\mu(z_1,\dots,z_{N-1}).\]
We remark that it is non-trivial that the left-hand side decomposes into the right-hand side in the case of $X=D$ since $g^{D_{N-1}}_\lambda$ is not irreducible. However, we have the above decomposition of $g^{D_N}_\lambda(z_1,\dots, z_{N-1},1)$ by \cite[Proposition 2.5]{CG20}. By Proposition \ref{prop:contact} and Lemma \ref{lem:yuki}, we have ${}_q\Lambda^{N}_{N-1}(\lambda, \,\cdot\,)\geq0$ and 
$\sum_{\mu\in\mathbb{S}_{N-1}^+}{}_q\Lambda^{N}_{N-1}(\lambda,\mu)=1$ for any $\lambda\in\mathbb{S}^+_N$.
Let $\mathcal{M}_p(\mathbb{S}_N^+)$ be the set of all probability measures on $\mathbb{S}_N^+$. Then we have $m\in\mathcal{M}_p(\mathbb{S}_{N}^+)\mapsto m{}_q\Lambda^{N}_{N-1}\in\mathcal{M}_p(\mathbb{S}_{N-1}^+)$ by
\[m{}_q\Lambda^{N}_{N-1}(\mu)=\sum_{\lambda\in\mathbb{S}_{N}^+}m(\lambda){}_q\Lambda^{N}_{N-1}(\lambda, \mu)\]
for any $\mu\in\mathbb{S}_{N-1}^+$. Thus, the $\mathcal{M}_p(\mathbb{S}^+_N)$ form a projective system, and an element in the projective limit $\varprojlim_N(\mathcal{M}_p(\mathbb{S}_N^+), {}_q\Lambda^N_{N-1})$ is called a \emph{coherent system} of the $BC$ type $q$-Gelfand-Tsetlin graph (see \cite{CG20} for the definition of the $BC$ type $q$-Gelfand--Tsetlin graph itself). It is clear that $\varprojlim_N(\mathcal{M}_p(\mathbb{S}_N^+), {}_q\Lambda^N_{N-1})$ is a convex set. Moreover, we equip it with the topology of component-wise weakly convergence. We also equip $\mathrm{Ch}(G_q(\infty))$ with the topology of point-wise convergence on $\mathfrak{A}$ (see \cite[Remark 3.2]{Sato3}). Then we obtain the main theorem in the paper.

\begin{theorem}\label{thm:main}
For $G_q(\infty)=SO_q(2\infty+1), Sp_q(\infty)$, there exists an affine homeomorphism from $\mathrm{Ch}(G_q(\infty))$ to $\varprojlim_N(\mathcal{M}_p(\mathbb{S}_N^+), {}_q\Lambda^N_{N-1})$ such that the coherent system $(m_N)_N$ corresponding to $\chi\in\mathrm{Ch}(G_q(\infty))$ is given as $\chi\Theta^N_\infty=\sum_{\lambda\in P(G_q(N))}m_N(\lambda)\chi_\lambda$ for any $N$.
\end{theorem}
\begin{proof}
It follows from the general statements \cite[Proposition 2.6]{Sato1} and \cite[Proposition 2.1]{Sato1}.
\end{proof}

The following is a consequence of Theorem \ref{thm:main} and \cite[Theorem 6.6]{CG20}.
\begin{corollary}\label{cor:para}
The extreme quantized characters of $G_q(\infty)=SO_q(2\infty+1), Sp_q(\infty)$ are completely parametrized by the set
$\mathfrak{Y}=\{(y_1,y_2,\dots)\in\mathbb{Z}_{\geq0}^\infty\mid 0\leq y_1\leq y_2\leq\cdots\}$.
\end{corollary}

We have the following by Theorem \ref{thm:main} and \cite[Corollary 3.1]{Sato3}. The latter is following from \cite[Proposition 7.20, Corollary 8.3(3)]{Ueda20}.
\begin{corollary}\label{cor:mado}
For $X=B, C$, the $\varprojlim_N (\mathcal{M}_p(\mathbb{S}^+_N), {}_q\Lambda^N_{N-1})$ is a Choquet simplex. Moreover, there exists a homeomorphism from $\varprojlim_N (\mathcal{M}_p(\mathbb{S}^+_N), {}_q\Lambda^N_{N-1})$ to $\mathcal{M}_p(\mathfrak{Y})$.
\end{corollary}

We give a possible analysis of $\mathrm{Ch}(SO_q(2\infty))$ based on the work by Cuenca and Gorin \cite{CG20}. Let $\mathcal{A}$ be the set of all quantized characters $\chi$ in $\mathrm{Ch}(SO_q(2\infty))$ such that for any $N$ and $\mu\in\mathbb{S}^+_N$
\[\frac{P^\chi_N(\lambda)}{f^{D_N}_\lambda(q^{N-1},\dots, q, 1)}=\frac{P^\chi_N(\tilde{\lambda})}{f^{D_N}_{\tilde{\lambda}}(q^{N-1},\dots, q,1)},\]
where $\chi\Theta^\infty_N=\sum_{\lambda\in P(SO_q(2N))}P^\chi_N(\lambda)\chi_\lambda\in\mathrm{Ch}(SO_q(2N))$. By a similar proof of Theorem \ref{thm:main} and Corollary \ref{cor:para}, we have the following:
\begin{proposition}
There exists the affine homeomorphism from $\mathcal{A}$ to $\varprojlim_N(\mathcal{M}_p(\mathbb{S}_N^+), {}_q\Lambda^N_{N-1})$ for $X=D$ such that the coherent system $(m_N)_N$ corresponding to $\chi\in \mathrm{Ch}(SO_q(2\infty))$ is given as $\chi\Theta^N_\infty=\sum_{\lambda\in P(G_q(N))}m_N(\lambda)\chi_\lambda$ for any $N$. Moreover, there exists a bijection between $\mathrm{ex}(\mathrm{Ch}(SO_q(2\infty)))$ and $\mathfrak{Y}$.
\end{proposition}

We comment on a description of extreme quantized characters. Any extreme quantized characters of inductive limit quantum groups relate to ergodic dynamical systems, and they can be approximated by irreducible quantized characters of finite rank quantum subgroups (see \cite[Section 2.5]{Sato1}). Actually, Cuenca and Gorin have done such an approximation, and hence we obtain explicit descriptions of extreme quantized characters of $SO_q(2\infty+1)$ and $Sp_q(\infty)$ (see \cite[Theorem 3.13]{CG20}). Moreover, quantized characters of $SO_q(2\infty)$ in $\mathrm{ex}(\mathcal{A})$ also admit an approximation approach. However, to the best of the author's knowledge, a representation-theoretic meaning of quantized characters in $\mathrm{ex}(\mathcal{A})$ and whether they relate to ergodic dynamical systems are unknown.

%%%%%%%%%%%%%%%%%%%%%%%%%%%%%%%%%%%%%%%%%%%%%%%%%%%%%%%%%%%%%%%%%%%%%%%%%%%%%%%%%%%%%%%%%%%%%%%%%%%%%%%%%%%%%%%%%%%%%%%%%%%%%%%%%%
\section{Markov semigroups on unitary duals}
In this section, we discuss $BC$ analogs of Markov semigroups that we study for the type $A$ in \cite{Sato4}. First, we propose distinguished quantized characters of $SO_q(2N+1)$, $Sp_q(N)$. Let $\Omega$ be the subset of $\omega=(\alpha,\beta,\gamma)\in\mathbb{R}_{\geq 0}^\infty\times\mathbb{R}_{\geq0}^\infty\times\mathbb{R}_{\geq0}$ satisfying that 
\[\alpha=(\alpha_1\geq\alpha_2\geq\cdots),\quad\beta=(\beta_1\geq\beta_2\geq\cdots),\quad \gamma\geq0,\quad \beta_1\leq1.\]
For any $\omega\in \Omega$ we define 
\[\Phi_\omega(x):=e^{\gamma(x-1)}\prod_{i=1}^\infty\frac{1+\beta_i(1-\frac{\beta_i}{2})(x-1)}{1-\alpha_i(1+\frac{\alpha_i}{2})(x-1)},\]
\[\Psi_\omega(z):=\Phi_\omega\left(\frac{z+z^{-1}}{2}\right)=e^{\gamma\left(\frac{z+z^{-1}}{2}-1\right)}\prod_{i=1}^\infty\frac{1+\frac{\beta_i}{2}(z-1)}{1-\frac{\alpha_i}{2}(z-1)}\frac{1+\frac{\beta_i}{2}(z^{-1}-1)}{1-\frac{\alpha_i}{2}(z^{-1}-1)}.\]
Then we define $\varphi_\omega(n)$ by $\Phi_\omega(x)=\sum_{n=0}^\infty\varphi_\omega(n)x^n$. We remark that for any $\omega\in \Omega$ the functions $\Psi_\omega$ give extreme characters of $SO(\infty)$ and $Sp(\infty)$. See \cite{OO06} for more details.

Let $G_q(N)=SO_q(2N+1), Sp_q(N)$ and $X=B, C$, respectively. We recall that the constants $c(X)_i$ are define by Equation \eqref{eq:daniels} for $i\geq1$, and let $\kappa_N\colon \mathbb{T}^N\to W^*(G_q(N))$ be as is in Section 4. For $i=1,\dots, N$ we define $q^{X}_i:=(q^{c(X)_i}+q^{-c(X)_i})/2$.
\begin{lemma}\label{lem:goro}
For any $N\geq1$ let $\omega\in \Omega$ such that $\Psi_\omega(q^{c(X)_i})>0$ for $i=1,\dots, N$ and the Taylor expansion of $\Phi_\omega(x)$ converges in annulus containing $q^{X}_1,\dots, q^{X}_N$. Then there exists the quantized character $\chi_{\omega, N}\in\mathrm{Ch}(G_q(N))$ such that
\[\chi_{\omega,N}(\kappa_N(z_1,\dots, z_N))=\prod_{i=1}^N\frac{\Psi_\omega(q^{c(X)_i}z_i)}{\Psi_\omega(q^{c(X)_i})}\]
for any $(z_1,\dots, z_N)\in\mathbb{T}^N$. 
\end{lemma}
\begin{proof}
Since any convex combinations of irreducible quantized characters define quantized characters (see \cite[Lemma 2.2]{Sato1}), by Proposition \ref{prop:contact}, it suffices to show that the right-hand side coincides with a convex combination of the functions of the form
\[g^{X_N}_\lambda(z_1,\dots,z_N):=\frac{f^{X_N}_\lambda(q^{c(X)_1}z_1,\dots,q^{c(X)_N}z_N)}{f^{X_N}_\lambda(q^{c(X)_1},\dots,q^{c(X)_N})}\]
for any $\lambda\in\mathbb{S}^+_N$. Let $x_i=(q^{c(X)_i}z_i+q^{-c(X)_i}z_i^{-1})/2$ for $i=1,\dots,N$. Then we have
\begin{align*}
V(x_1,\dots, x_N)\prod_{i=1}^N\Psi_\omega(q^{c(X)_i}z_i)
&=\sum_{\sigma\in S(N)}\sum_{n_1,\dots,n_N\geq0}\mathrm{sgn}(\sigma)\prod_{i=1}^N\varphi_\omega(n_i)x_i^{n_j+N-\sigma(i)}\\
&=\sum_{n_1,\dots,n_N\geq0}\det[\varphi_\omega(n_j+i)]_{i, j=1}^N\prod_{i=1}^Nx_i^{n_i+N}\\
&=\sum_{n_1>\dots>n_N\geq0}\sum_{\sigma\in S(N)}\det[\varphi_\omega(n_{\sigma(j)}+i)]_{i, j=1}^N\prod_{i=1}^Nx_i^{n_{\sigma(i)}+N}\\
&=\sum_{n_1>\dots>n_N\geq0}\det[\varphi_\omega(n_j+i)]_{i, j=1}^N\det\left[x_i^{n_{j}+N}\right]_{i, j=1}^N\\
%&=\sum_{\lambda\in\mathbb{S}^+_N}\det[\varphi_\omega(\lambda_j-j+i)]_{i, j=1}^N\det\left[x_i^{\lambda_j-j+N}\right]_{i, j=1}^N.
\end{align*}
Therefore, we have 
\[\prod_{i=1}^N\Psi_\omega(q^{c(X)_i}z_i)=\sum_{\lambda\in\mathbb{S}^+_N}\det[\varphi_\omega(\lambda_j-j+i)]_{i, j=1}^N\frac{\det\left[x_i^{\lambda_j-j+N}\right]_{i, j=1}^N}{V(x_1,\dots,x_N)}.\]
Let $\mathfrak{P}_\lambda(x_1,\dots, x_N; a, b)$ be the multivariate Jacobi polynomial with the label $\lambda$ (see \cite[Section 7]{OO06}), where $a$, $b$ are parameters given as $a=\pm1/2$, $b=1/2$ for $X=B, C$, respectively. Then we have
\[\frac{\det\left[x_i^{\lambda_j-j+N}\right]_{i, j=1}^N}{V(x_1,\dots,x_N)}=\frac{\mathfrak{P}_\lambda(x_1,\dots, x_N; a, b)}{\prod_{i=1}^N\varkappa(\lambda_i-i+N; a, b)},\]
where 
\[\varkappa(l; a, b):=\frac{\Gamma(2l+a+b+1)}{2^ll!\Gamma(l+a+b+1)}.\]
Moreover, by \cite[Proposition 7.1]{OO06}, we have
\[g^{X_N}_\lambda(z_1,\dots, z_N)=\frac{\mathfrak{P}_\lambda(x_1,\dots, x_N; a, b)}{\mathfrak{P}_\lambda(q^{X}_1,\dots, q^{X}_N; a, b)},\]
Therefore, we have 
\[\prod_{i=1}^N\Psi_\omega(q^{c(X)_i}z_i)=\sum_{\lambda\in\mathbb{S}^+_N}\det[\varphi_\omega(\lambda_j-j+i)]_{i, j=1}^N\frac{\mathfrak{P}_\lambda(q^{X}_1,\dots, q^{X}_N; a, b)}{\prod_{i=1}^N\varkappa(\lambda_i-i+N; a, b)}g^{X_N}_\lambda(z_1,\dots,z_N).\]
Then, by definition, $\prod_{i=1}^N\varkappa(\lambda_i-i+N; a, b)>0$. By \cite[Proposition 7.1]{OO06}, $\mathfrak{P}_\lambda(q^{X_N}_1,\dots, q^{X_N}_N; a, b)$ is equal to $f^{X_N}_\lambda(q^{c(X_N)_1},\dots, q^{c(X_N)_N})=\mathrm{Tr}_{V_\lambda}(T_\lambda(K_{2\rho}))$ times a positive constant. Thus, we have $\mathfrak{P}_\lambda(q^{X}_1,\dots, q^{X}_N; a, b)>0$. Moreover, since $\Phi_\omega(x)$ is a special case of the Voiculescu functions, it is known that $\det[\varphi_\omega(\lambda_j-j+i)]_{i, j=1}^N\geq0$ for any $N\geq1$ and $\lambda\in\mathbb{S}^+_N$. Therefore, $\prod_{i=1}^N\Psi_\omega(q^{c(X)_i}z_i)$ is a linear combination of the $g^{X_N}_\lambda(z_1,\dots, z_N)$ with nonnegative coefficients, and the summation of coefficients is equal to $\prod_{i=1}^N\Psi_\omega(q^{c(X)_i})$. Namely, $\prod_{i=1}^N\Psi_\omega(q^{c(X)_i}z_i)/\Psi_\omega(q^{c(X)_i})$ is a convex combination of the $g^{X_N}_\lambda(z_1,\dots, z_N)$.
\end{proof}

By the above lemma, we have the decomposition 
\[\chi_{\omega, N}(\kappa_N(z_1,\dots,z_N))=\sum_{\lambda\in\mathbb{S}^+_N}P_{\chi_{\omega, N}}(\lambda)g^{X_N}_\lambda(z_1,\dots,z_N),\]
where $P_{\chi_{\omega, N}}(\lambda)\geq0$ and $\sum_{\lambda\in\mathbb{S}^+_N}P_{\chi_{\omega, N}}(\lambda)=1$. By the above proof, for any $\lambda\in\mathbb{S}^+_N$ we have
\[P_{\chi_{\omega, N}}(\lambda)=\frac{\det[\varphi_\omega(\lambda_j-j+i)]_{i, j=1}^N}{\prod_{i=1}^N\Psi_\omega(q^{c(X)_i})}\frac{\mathfrak{P}_\lambda(q^{X}_1,\dots, q^{X}_N; a, b)}{\prod_{i=1}^N\varkappa(\lambda_i-i+N; a, b)}.\]

Let $G_q(\infty)=SOq(2\infty+1), Sp_q(\infty)$ and $\Theta^\infty_N\colon W^*(G_q(N))\to W^*(G_q(\infty))$ as is in Section 5. The following is a consequence of the above lemma.
\begin{proposition}\label{prop:char_infty}
Let $\omega\in \Omega$ such that $\Psi_\omega(q^{c(X)_i})>0$ for any $i\geq1$ and the Taylor expansion of $\Phi_\omega(x)$ converges in annulus containing $q^X_i$ for any $i\geq1$. Then there exists a quantized character $\chi_\omega\in\mathrm{Ch}(G_q(\infty))$ such that $\chi_\omega\Theta^\infty_N=\chi_{\omega, N}$ for any $N\geq1$.
\end{proposition}
\begin{proof}
By assumption, for any $N$ the functions $\prod_{i=1}^N\Psi_\omega(q^{c(X)_i}z_i)/\Psi_\omega(q^{c(X)_i})$ give quantized characters $\chi_{\omega, N}\in\mathrm{Ch}(G_q(N))$. Moreover, it is clear that the sequence $(\chi_{\omega, N})_N$ gives a coherent system of the $BC$ type $q$-Gelfand--Tsetlin graph. Therefore, by Theorem \ref{thm:main}, there exists a quantized character $\chi_\omega\in\mathrm{Ch}(G_q(\infty))$ such that $\chi_\omega\Theta^\infty_N=\chi_{\omega, N}$ for any $N\geq1$.
\end{proof}

For any quantized character of $G_q(N)$ we obtain Markov semigroups on $\widehat{G_q(N)}\cong \mathbb{S}^+_N$ applying a general theory in \cite{Sato4}.
\begin{theorem}\label{thm:markov}
For any $N\geq1$ let $\omega\in \Omega$ satisfying the assumption in Lemma \ref{lem:goro}. Then there exists a Markov semigroup on $\mathbb{S}^+_N$ with the generator $\mathbb{L}_{\chi_{\omega, N}}$ given by
\begin{align*}
\mathbb{L}_{\chi_{\omega, N}}(\lambda, \mu)
=\frac{f_\mu(q^{c(X)_1},\dots, q^{c(X)_N})}{f_\lambda(q^{c(X)_1},\dots, q^{c(X)_N})}\frac{\det\left[\oint_{|z|=1}\Psi_\omega(z)(z^{l_i}-z^{-l_i})\overline{(z^{m_j}-z^{-m_j})}\frac{dz}{2\pi\mathrm{i}z}\right]_{i, j=1}^N}{2^N\prod_{j=1}^N\Psi_\omega(q^{c(X_N)_1})}-\delta_{\lambda, \mu}
\end{align*}
for any $\lambda, \mu\in\mathbb{S}^+_N$, where $l_i:=\lambda_i+c(X)_{N-i+1}$ and $m_i:=\mu_i+c(X)_{N-i+1}$ for $i=1,\dots, N$.
\end{theorem}
\begin{proof}
By \cite[Corollary 5.2]{Sato4}, there is a Markov semigroup on $\mathbb{S}^+_N$ with the generator $\mathbb{L}$ given by
\[\mathbb{L}(\lambda, \mu)=\frac{f_\mu(q^{c(X)_1},\dots, q^{c(X)_N})}{f_\lambda(q^{c(X)_1},\dots, q^{c(X)_N})}\sum_{\nu\in\mathbb{S}^+_N}\frac{P_{\chi_{\omega,N}}(\nu)N_{\lambda,\nu}^\mu}{f_\nu(q^{c(X)_1},\dots, q^{c(X)_N})}-\delta_{\lambda, \mu}\]
for any $\lambda, \mu\in\mathbb{S}^+_N$, where $N^\mu_{\lambda, \nu}$ is the multiplicity of the irreducible representation associated with $\mu$ in the tensor product of two irreducible representations associated with $\lambda$ and $\nu$. Then, by the Weyl integration formula (see, e.g. \cite{FH:book}), we have
\begin{align*}
\mathbb{L}(\lambda,\mu)+\delta_{\lambda, \mu}
=\frac{f_\mu(q^{c(X)_1},\dots, q^{c(X)_N})}{f_\lambda(q^{c(X)_1},\dots, q^{c(X)_N})}\int_{\mathbb{T}^N}\prod_{j=1}^N\frac{\Psi_\omega(z_j)}{\Psi_\omega(q^{c(X)_j})}f_\lambda(z_1,\dots, z_N)\overline{f_\mu(z_1,\dots, z_N)}m_{a,b}(dz),
\end{align*}
where 
\[m_{a, b}(dz)=\frac{1}{N!2^N}\prod_{1\leq i<j\leq N}|z_i-z_j|^2|1-z_iz_j|^2\prod_{j=1}^N|1-z_j|^{2a+1}|1+z_j|^{2b+1}\frac{dz_j}{2\pi\mathrm{i}z_j}.\]
Moreover, we have
\begin{align*}
&\int_{\mathbb{T}^N}\prod_{j=1}^N\Psi_\omega(z_j)f_\lambda(z_1,\dots, z_N)\overline{f_\mu(z_1,\dots, z_N)}m_{a, b}(dz)\\
&=\frac{1}{N!2^N}\int_{\mathbb{T}^N}\prod_{j=1}^N\Psi_\omega(z_j)\det[z_j^{l_i}-z_j^{-l_i}]_{i, j=1}^N\overline{\det[z_j^{m_i}-z_j^{-m_i}]_{i, j=1}^N}\prod_{j=1}^N\frac{dz_j}{2\pi\mathrm{i}z_j}\\
&=\frac{1}{2^N}\det\left[\oint_{|z|=1}\Psi_\omega(z)(z^{l_i}-z_j^{-l_i})\overline{(z^{m_j}-z^{-m_j})}\frac{dz}{2\pi\mathrm{i}z}\right]_{i, j=1}^N.
\end{align*}
Therefore, we have $\mathbb{L}_{\chi_{\omega, N}}=\mathbb{L}$.
\end{proof}

We recall that there exists a homeomorphism from $\mathcal{M}_p(\mathfrak{Y})$ to $\varprojlim_N (\mathcal{M}_p(\mathbb{S}^+_N), {}_q\Lambda^N_{N-1})$ by Corollary \ref{cor:mado}. Thus, for any $N\geq1$ there exists a mapping ${}_q\Lambda^\infty_N\colon\mathcal{M}_p(\mathfrak{Y})\to\mathcal{M}_p(\mathbb{S}^+_N)$ such that the homeomorphism is given as $P\in\mathcal{M}_p(\mathfrak{Y})\mapsto (P{}_q\Lambda^\infty_N)_N\varprojlim_N (\mathcal{M}_p(\mathbb{S}^+_N), {}_q\Lambda^N_{N-1})$.

\begin{corollary}
Let $\omega\in\Omega$ satisfying the assumption in Proposition \ref{prop:char_infty}. Let $(Q^{\chi_{\omega, N}}_t)_{t\geq0}$ be the Markov semigroup on $\mathbb{S}^+_N$ in Theorem \ref{thm:markov}. Then there exists a unique Markov semigroup $(Q^{\chi_\omega}_t)_{t\geq0}$ on $\mathfrak{Y}$ such that $Q^{\chi_\omega}_t{}_q\Lambda^\infty_N={}_q\Lambda^\infty_NQ^{\chi_\omega, N}_t$ for any $N\geq1$ and $t\geq0$. 
\end{corollary}
\begin{proof}
By \cite[Proposition 8.1]{Sato4}, we have $\Lambda^N_{N-1}Q^{\chi_{\omega, N-1}}_t=Q^{\chi_{\omega, N}}\Lambda^N_{N-1}$ for any $N$ and $t\geq0$. Thus, the statement follows from \cite[Proposition 2.4]{BO}. 
\end{proof}

\begin{remark}
For any $\omega\in\Omega$ satisfying the assumption in Proposition \ref{prop:char_infty}, we can also construct the Markov semigroup $(Q^{\chi_\omega}_t)_{t\geq0}$ directly using the quantized character $\chi_\omega\in\mathrm{Ch}(G_q(\infty))$. See \cite[Definition 8.4]{Sato4}.
\end{remark}

At the end of the paper, we give a simple example of Markov semigroup on $\mathbb{S}^+_N$.
\begin{example}
Let $\gamma>0$ and $\alpha=\beta=(0,0,\dots)$. Then $\omega=(\alpha, \beta, \gamma)\in\Omega$ satisfies the assumption in Lemma \ref{lem:goro} for any $N$ since $\Psi_\omega(z)=e^{-\gamma}e^{\gamma(z+z^{-1})/2}$. Then for any integer $m\in\mathbb{Z}$ we have
\[\oint_{|z|=1}\Psi_\omega(z)z^{-m}\frac{dz}{2\pi\mathrm{i}z}=e^{-\gamma}I_m(\gamma),\]
where $I_m$ is the modified Bessel function of the first kind. Then, by Theorem \ref{thm:markov}, we obtain a Markov semigroup on $\mathbb{S}^+_N$ with the generator $\mathbb{L}_\gamma$ given by
\[\mathbb{L}_\gamma(\lambda, \mu)=\frac{f_\mu(q^{c(X)_1},\dots, q^{c(X)_N})}{f_\lambda(q^{c(X)_1},\dots, q^{c(X)_N})}\frac{\det\left[I_{-l_i+m_j}(\gamma)+I_{l_i-m_j}(\gamma)-I_{l_i+m_j}(\gamma)-I_{-l_i-m_j}(\gamma)\right]_{i, j=1}^N}{2^Ne^{\gamma N}\prod_{j=1}^N\Psi_\omega(q^{c(X_N)_1})}-\delta_{\lambda, \mu}\]
for any $\lambda, \mu\in\mathbb{S}^+_N$, where $l_i:=\lambda_i+c(X)_{N-i+1}$ and $m_i:=\mu_i+c(X)_{N-i+1}$ for $i=1,\dots, N$.
\end{example}

\section*{Acknowledgment}
The author gratefully acknowledges the valuable comments from his supervisor, Professor Yoshimichi Ueda. This work was supported by JSPS Research Fellowship for Young Scientists (KAKENHI Grant Number JP 19J21098).
}

%%%%%%%%%%%%%%%%%%%%%%%%%%%%%%%%%%%%%%%%%%%%%%%%%%%%%%%%%%%%%%%%%%%%%%%%%%%%%%%%%%%%%%%%%%%%%%%%%%%%%%%%%%%%%%%%%%%%%%%%%%%%%%%%%%%

\end{document}